\documentclass[10pt,a4paper,reqno]{amsart}

\usepackage{amsmath}
\usepackage{amsfonts}
\usepackage{amssymb}
\usepackage{amsthm}
\usepackage{mathrsfs}
\usepackage[shortlabels]{enumitem}
\usepackage{graphicx}
\usepackage{color}
\usepackage{dsfont}
\usepackage[latin1]{inputenc}
\usepackage{MnSymbol}
\usepackage{enumitem}
\usepackage{extpfeil}
\usepackage{mathtools}
\usepackage{url}
\usepackage[margin=1.25in]{geometry}
\usepackage{fancyhdr}
\usepackage[all,cmtip]{xy}
\usepackage{hyperref}
\usepackage{tikz}
\usepackage{thm-restate}
\usepackage{amscd}

\numberwithin{equation}{section}

\newcommand{\C}{{\mathds C}}
\newcommand{\Q}{{\mathds Q}}
\newcommand{\R}{{\mathds R}}

\newcommand{\Z}{{\mathds Z}}

\newcommand{\FP}{{\rm FP}}

\def\ker{{\rm{ker}}}
\def\im{{\rm{im}}}
\def\dim{{\rm{dim}}}

\theoremstyle{plain}
\newtheorem{theorem}{Theorem}[section]

\newtheorem{corollary}[theorem]{Corollary}

\newtheorem{lemma}[theorem]{Lemma}
\newtheorem{proposition}[theorem]{Proposition}
\newtheorem{question}[theorem]{Question}
\newtheorem*{question*}{Question}
\newtheorem{addendum}[theorem]{Addendum}

\theoremstyle{definition}
\newtheorem{remark}[theorem]{Remark}
\newtheorem*{acknowledgements*}{Acknowledgements}

\newtheorem{definition}[theorem]{Definition}
\newtheorem*{notation*}{Notation}
\newtheorem*{convention*}{Convention}

\title{Groups with exotic finiteness properties from complex Morse theory}
\author{Claudio Llosa Isenrich}
\address{Faculty of Mathematics, KIT, Englerstr. 2, 76131 Karlsruhe, Germany}
\email{claudio.llosa@kit.edu}
\author{Pierre Py}
\address{Institut Fourier, Universit\'e Grenoble Alpes \& CNRS, 38000 Grenoble, France}
\email{pierre.py@univ-grenoble-alpes.fr}

\thanks{The first author is grateful to the RTG 2229 ``Asymptotic Invariants and Limits of Groups and Spaces'' for their support.}

\begin{document}

\begin{abstract}
Recent constructions have shown that interesting behaviours can be observed in the finiteness properties of K\"ahler groups and their subgroups. In this work, we push this further and exhibit, for each integer $k$, new hyperbolic groups admiting surjective homomorphisms to $\Z$ and to $\Z^{2}$, whose kernel is of type $\mathscr{F}_{k}$ but not of type $\mathscr{F}_{k+1}$. By a fibre product construction, we also find examples of nonnormal subgroups of K\"ahler groups with exotic finiteness properties.  
\end{abstract}

\maketitle

\section{Introduction}
Two of the most basic properties of groups are being finitely generated and being finitely presented. These properties admit a geometric interpretation in terms of classifying spaces, which leads to a higher dimensional generalisation introduced by Wall \cite{Wal-65}: for a natural number $n$, a group $G$ is called of finiteness type $\mathscr{F}_n$ if it admits a $K(G,1)$ which is a CW-complex with finite $n$-skeleton. Finite generation is then equivalent to $\mathscr{F}_1$ while finite presentability is equivalent to $\mathscr{F}_2$. We say that a group has \emph{exotic} finiteness properties, if it is $\mathscr{F}_n$, but not $\mathscr{F}_{n+1}$ for some integer $n\geq 0$. The existence of such groups is classical for $n=0, 1$ and was proved by Stallings \cite{Sta-63} for $n=2$ and by Bieri for all $n\geq 3$ \cite{Bie-76}. Since then, many examples have been constructed, showcasing that exotic finiteness properties can appear under a wide range of additional assumptions on the group. 

Classical methods used to construct groups with exotic finiteness properties include Bestvina--Brady Morse theory \cite{BesBra-97} and Brown's criterion \cite{Bro-87}. For recent use of the latter criterion, see~\cite{SWZ19} and the references there. Starting with the works of Kapovich \cite{Kap-98} and Dimca, Papadima and Suciu \cite{DimPapSuc-09-II}, it has become increasingly apparent that complex Morse theory provides a powerful method for constructing groups with exotic finiteness properties. The purpose of these works was to prove that fundamental groups of compact K\"ahler manifolds (\emph{K\"ahler groups}) can be non-coherent, respectively can have arbitrary exotic finiteness properties. 

These methods have since been extended to produce a range of examples of K\"ahler groups with exotic finiteness properties~\cite{Llo-16-II, Llo-17, BriLlo-16, NicPy-21}. These works generalised the construction in \cite{DimPapSuc-09-II}, leading to K\"ahler groups with exotic finiteness properties which all arise as fundamental groups of generic fibres of holomorphic maps from certain compact K\"ahler manifolds onto a complex torus. However, the work of Nicol\'as and Py~\cite{NicPy-21} provides tools for constructing such examples from holomorphic maps with isolated singularities onto arbitrary closed Riemann surfaces (possibly of genus greater than $1$), showing that the potential of these methods stretches beyond the realm of the already known examples. 

Finally, using again complex Morse theory, we recently produced, for every integer $n\geq 0$, examples of subgroups of hyperbolic groups of type $\mathscr{F}_n$, but not $\mathscr{F}_{n+1}$~\cite{LloPy-22}. These were the first such examples when $n\ge 4$. This solved an old problem raised by Brady \cite{Bra-99}, where classical methods using Bestvina--Brady Morse theory had so far only been able to provide an answer up to $n=3$ \cite{Bra-99,Lod-18,Kro-21,LIMP-21}. This showed that indeed the methods from complex Morse theory are sufficiently powerful to construct examples of groups with exotic finiteness properties that are of interest beyond the realm of complex geometry. Let us also mention in this context that recently the combination of Bestvina--Brady Morse theory and real hyperbolic geometry led to the first example of a non-hyperbolic subgroup of a hyperbolic group with a finite classifying space \cite{IMM-22, IMM-21}.

In this work we provide further examples of groups with exotic finiteness properties built from complex geometry. We emphasise that we produce both K\"ahler and non-K\"ahler groups. The first result takes as input a new class of hyperbolic K\"ahler groups constructed by Stover and Toledo \cite{StoTol-21-II}. Their groups arise as fundamental groups of certain compact K\"ahler manifolds which admit a K\"ahler metric of negative sectional curvature, but are not homotopy equivalent to any locally symmetric manifold. Combining ideas from~\cite{LloPy-22} and~\cite{StoTol-21-II}, we shall prove:

\begin{theorem}\label{thm:ST-Intro}
    For every $n\geq 2$ there is an $n$-dimensional compact K\"ahler manifold $Y$ which admits a K\"ahler metric of negative sectional curvature, is not homotopy equivalent to any locally symmetric manifold and which has the following property. There exists a dense open set $O\subset H^{1}(Y,\R)-\{0\}$ which is invariant by multiplication by nonzero scalars, such that for any homomorphism $\phi: \pi_1(Y)\to \mathbb{Z}$ contained in $O$, the kernel $\ker(\phi)$ is of type $\mathscr{F}_{n-1}$ but not of type ${\rm FP}_n(\mathbb{Q})$.
\end{theorem}

The definition of the finiteness property ${\rm FP}_n(\mathbb{Q})$ will be recalled in Section~\ref{sec:FinProps}. This produces many new subgroups of hyperbolic groups of finiteness type $\mathscr{F}_{n-1}$ and not $\mathscr{F}_n$, thus extending the main result from~\cite{LloPy-22}. Using the Cartwright--Steger surface, we can also produce further examples of K\"ahler groups of type $\mathscr{F}_{2n-1}$ and not $\mathscr{F}_{2n}$ for all $n\geq 2$. See Section \ref{sec:ST-CS} for a discussion.

Besides producing homomorphisms from certain negatively curved K\"ahler groups onto $\mathbb{Z}$, whose kernels have exotic finiteness properties, we can also produce similar homomorphisms onto $\mathbb{Z}^{2}$, using results from the theory of {\it Bieri-Neumann-Strebel-Renz invariants} (in short BNSR-invariants; see Section~\ref{sec:FinProps} for some background). This is the content of our next main result. 

\begin{theorem}\label{thm:zeetwo} Let $X$ be a closed aspherical K\"ahler manifold with positive first Betti number and nonzero Euler characteristic. Assume that the Albanese map of $X$ is finite. Let $n={\rm dim}_{\C} \, X$. Then there exist surjective morphisms $\pi_{1}(X)\to \mathbb{Z}^{2}$ whose kernel is of type $\mathscr{F}_{n-1}$ but not of type ${\rm FP}_n(\mathbb{Q})$. The set of such homomorphisms is open. 
\end{theorem}

Let us recall that a map $f$ is said to be {\it finite} if the preimage by $f$ of each point of the target space is a finite set. For the definition of the Albanese torus and Albanese map of a closed K\"ahler manifold, we refer the reader to~\cite[\S 3.1]{LloPy-22} and~\cite[\S 12.1.3]{Voi-book}. As for the topology alluded to in the statement of Theorem~\ref{thm:zeetwo}, it will be defined in~Section~\ref{sec:FinProps}. As we shall explain in Section~\ref{sec:zeetwo}, this theorem applies to certain arithmetic ball quotients as well as to some of the manifolds built by Stover and Toledo in \cite{StoTol-21-II}. This allows us to deduce the following result.

\begin{corollary}\label{cor:zeetwo}
For every $k\ge 1$ there is a hyperbolic group $G$ and a surjective homomorphism $\phi: G \to \mathbb{Z}^2$ such that $\ker(\phi)$ is of type $\mathscr{F}_{k}$, but not ${\rm FP}_{k+1}(\mathbb{Q})$.
\end{corollary}

We now proceed to explain a new way of using complex Morse theory to produce groups with exotic finiteness properties as fibre products. An important novelty is that these groups are not constructed as kernels of homomorphisms, making them rather different from most other groups constructed using Morse theory. One main consequence of this approach is the following result. 

\begin{theorem}\label{thm:Livne}
Let $X_{1}=\Gamma\backslash \mathbb{B}_{\mathbb{C}}^n$ be a compact complex ball quotient with $n\ge 2$, and let $p_1:X_1\to \Sigma$ be a surjective holomorphic map with connected fibres onto a closed hyperbolic Riemann surface. Assume that $p_1$ has a finite non-empty set of critical points. Let $p_2: X_2 \to \Sigma$ be a ramified covering with non-trivial set of singular values that is disjoint from the set of singular values of $p_1$. Assume that $p_{2\ast}: \pi_1(X_2)\to \Sigma$ is surjective.

 Then the group theoretic fibre product $P\leq \pi_1(X_1)\times \pi_1(X_2)$ of the induced surjective homomorphisms $p_{i,\ast}:\pi_1(X_i)\to \pi_1(\Sigma)$, $i=1,~2$, is a non-normal K\"ahler subgroup of finiteness type $\mathscr{F}_n$ and not ${\rm FP}_{n+1}(\mathbb{Q})$.
\end{theorem}

We recall that the fibre product $P$ is the subgroup of $\pi_1(X_1)\times \pi_1(X_2)$ defined as follows:
$$P:=\{(a,b)\in \pi_1(X_1)\times \pi_1(X_2): p_{1,\ast}(a)=p_{2,\ast}(b)\}.$$
Concrete examples to which the $n=2$ version of Theorem \ref{thm:Livne} can be applied are the so-called {\it Livn\'e surfaces}~\cite{Liv-81}. We refer to Section \ref{sec:Liv-Ex} for their definition. For $n\ge 3$ we do not know examples of ball quotients admitting a map $p_1$ as in the theorem. We also observe that our assumption that the critical set of $p_1$ is non-empty is always satisfied, thanks to a theorem due to Koziarz and Mok~\cite{KozMok-2010}. We prefer however to state Theorem~\ref{thm:Livne} as above to emphasise the fact that we do need some critical points!

Finally, let us mention that our methods from \cite{LloPy-22} can also be applied to obtain a new proof of the following result of Kochloukova and Vidussi from \cite{KocVid-22}.

\begin{theorem}[{Kochloukova--Vidussi \cite[Corollary 1.11]{KocVid-22}}]\label{thm:KV}
    Let $n\geq 2$. There is an aspherical smooth complex projective variety $X^n$ of dimension $n$ whose fundamental group $\pi_1(X^n)$ is an irreducible polysurface group which contains for every $j\in \left\{1,\dots,n\right\}$ a subgroup of type $\mathscr{F}_{j-1}$, but not of type ${\rm FP}_{j}(\mathbb{Q})$.
\end{theorem}

The definition of polysurface groups will be recalled in Section~\ref{sec:Koc-Vid}, where we will also explain our new proof of Theorem~\ref{thm:KV}.

\subsection*{Structure}

In Section \ref{sec:FinProps}, we give some background on finiteness properties in group theory. In Section \ref{sec:CxTori}, we introduce the main construction methods of groups with exotic finiteness properties (K\"ahler or not) from maps onto complex tori. It can serve as an introductory reference for gaining an overview of the techniques from the works~\cite{DimPapSuc-09-II,Kap-98,Llo-16-II, Llo-17, BriLlo-16, NicPy-21, LloPy-22}. In Section \ref{sec:ST-Examples}, we illustrate these methods by proving Theorem \ref{thm:ST-Intro}. In Section~\ref{sec:zeetwo}, we prove Theorem~\ref{thm:zeetwo}. In Section~\ref{sec:Livne}, we explain how complex Morse theory can be used to produce non-normal subgroups with exotic finiteness properties, proving Theorem \ref{thm:Livne}. In Section \ref{sec:Koc-Vid}, we describe a new proof of Theorem \ref{thm:KV}, which was first proved by Kochloukova and Vidussi. Finally, Section~\ref{morse-theoretic-remarks} contains a few remarks about the existence of {\it perfect circle-valued Morse functions} and about alternative proofs of some of our results, relying purely on (real) Morse theory rather than on the theory of BNSR-invariants. 

\subsection*{Acknowledgements} We would like to thank Beno\^it Claudon and Vincent Koziarz who told us about the existence of fibrations with isolated critical points on Livn\'e's surfaces, and Bruno Martelli who pointed out to us the connection of our work to the existence of perfect circle-valued Morse functions. Finally, we thank again Beno\^it Claudon, for pointing out the reference~\cite{GrKi}. 


\section{Finiteness properties of groups}\label{sec:FinProps}

We already introduced the homotopical finiteness property $\mathscr{F}_n$, which requires that a group has a classifying space with finite $n$-skeleton. A second important set of finiteness properties are the homological finiteness properties. For an abelian unital ring $R$ we say that a group $G$ is of finiteness type $\FP_n(R)$ if there is a projective resolution
\[
    \dots \to P_n\to P_{n-1}\to \dots \to P_0\to R\to 0
\]
of the trivial $RG$-module $R$ which is finitely generated up to dimension $n$. Taking the free resolution induced by the cellular complex associated with a classifying space, it is easy to see that $\mathscr{F}_n$ implies property $\FP_n(R)$ for every $R$. In degree $1$, this is an equivalence: a group $G$ is of type $\FP_1(R)$ for some ring $R$ if and only if it is of type $\mathscr{F}_1$.  However, in higher dimensions Bestvina and Brady~\cite{BesBra-97} have shown that $\FP_n(R)$ does not imply $\mathscr{F}_n$ and, moreover, for different rings $R_1$ and $R_2$ the properties $\FP_n(R_1)$ and $\FP_n(R_2)$ are in general not equivalent. Finally, let us also mention that properties $\FP_n(\mathbb{Z})$ and $\mathscr{F}_2$ together imply $\mathscr{F}_n$. For a detailed introduction to finiteness properties, we refer the reader to \cite{Bro-82}.

An important source of examples of groups with exotic finiteness properties are kernels of homomorphisms onto free abelian groups. A key reason for this is that the homotopical finiteness properties of such kernels can be studied via Morse theoretical means. These properties are completely encoded by the so-called BNSR-invariants which we now introduce~\cite{BNS-87,BieRen-88,Renz-thesis}. There are also homological analogues of these invariants, but we shall not deal with them here.

The character sphere of a finitely generated group $G$ is the sphere
\[
 S(G):= \left( {\rm Hom}(G,\mathbb{R})-\left\{0\right\}\right)/\sim,
\]
where the equivalence relation $\sim$ is defined as follows. Two nonzero characters $\chi_1,~\chi_2: G\to \mathbb{R}$ are \emph{equivalent} if there is a real number $\lambda>0$ with $\chi_1=\lambda \cdot \chi_2$. For a group $G$ of type $\mathscr{F}_n$, one can define $n$ BNSR-invariants, which are subsets of the sphere $S(G)$. They are denoted by $\Sigma^{i}(G)$ ($1\le i \le n$) and form a decreasing sequence:
\[
  \Sigma^n(G)\subseteq \Sigma^{n-1}(G)\subseteq \dots \subseteq \Sigma^1(G)\subseteq S(G). 
\]

\noindent When $G$ admits a finite classifying space, the invariant $\Sigma^n (G)$ is defined for any natural integer $n$. The definition of these sets is related to the relative connectivity properties of certain ``half-spaces" associated to the characters, in the universal cover of a $K(G,1)$. Their precise definition is slightly technical. Since we will not work with it here, we omit it and refer to \cite{BNS-87, BieRen-88, BS-book}. We simply make two remarks:

$\bullet$ The invariant $\Sigma^1 (G)$, for $G$ a finitely generated group, can be defined quite simply as follows. One considers a Cayley graph $\Gamma$ associated to a finite symmetric generating subset of $G$, and a nonzero character $\chi : G \to \R$. One declares that $[\chi]\in \Sigma^{1}(G)$ if the subgraph of $\Gamma$ generated by the vertices where $\chi \ge 0$ is connected. See~\cite{BS-book} or~\cite[Ch. 11]{Py-book} for more details on this definition.   

$\bullet$ When $G$ is the fundamental group of a closed aspherical manifold $X$, any (nonzero) character $\chi$ is obtained by integration of a closed $1$-form $\alpha$ on $X$. On the universal cover $\widehat{X}$ of $X$, the pull-back of $\alpha$ is exact, and we can fix a primitive $f : \widehat{X}\to \R$ for it. The condition $[\chi] \in \Sigma^{k}(G)$ is then a condition on the behavior of the inclusion maps 
$$\{f\ge C\}\to \{ f \ge D\}$$
(for real numbers $C\ge D$) on homotopy groups in dimension $\le k-1$. See Appendix B in~\cite{BS-book} or Definition 12 in~\cite{LloPy-22}.  

The main properties that we will require here are summarised by the following results~\cite{BNS-87, BieRen-88, BS-book}.

\begin{proposition}\label{prop:BNSR}
    Let $G$ be a group of type $\mathscr{F}_n$. Then the following hold:
    \begin{enumerate}
        \item for $0\leq \ell \leq n$ the BNSR-invariant $\Sigma^{\ell}(G)$ is an open subset of $S(G)$;
        \item for $\chi: G\to \mathbb{Z}$ an integer-valued character, $[ \chi]\in \Sigma^{\ell}(G)\cap -\Sigma^{\ell}(G)$ if and only if $\ker(\chi)$ is of type $\mathscr{F}_{\ell}$.
    \end{enumerate}
\end{proposition}

\begin{theorem}\label{thm:finiteness-high-rank} Let $k\ge 1$. Let $G$ be a group of type $\mathscr{F}_{n}$ and let $\chi : G \to \Z^k$ be a surjective homomorphism. Then the kernel of $\chi$ is of type $\mathscr{F}_{n}$ if and only if for every nonzero homomorphism $u : \Z^k \to \R$, $[u\circ \chi]\in \Sigma^{n}(G)$.
\end{theorem}

Let us now elaborate on Theorem~\ref{thm:finiteness-high-rank} and introduce the topology on the space of surjective morphisms $G \to \Z^k$ that was alluded to in the introduction. Let $\chi$ and $G$ be as in Theorem~\ref{thm:finiteness-high-rank}. The classes
$$[u\circ \chi],$$
with $u\in {\rm Hom}(\mathbb{Z}^k,\mathbb{R})-\{0\}$, form a $(k-1)$-dimensional subsphere $S(\chi)$ of $S(G)$. Theorem~\ref{thm:finiteness-high-rank} then says that the kernel of $\chi$ is of type $\mathscr{F}_{n}$ if and only if the sphere $S(\chi)$ is contained in $\Sigma^{n}(G)$. Note that $S(\chi)$ depends only on the kernel of $\chi$ and it determines that kernel. Hence if $N$ is a normal subgroup of $G$ such that the quotient $G/N$ is isomorphic to $\Z^k$, we will write $S(N)$ for the sphere $S(\chi)$  where $\chi : G \to \Z^k$ is any surjection obtained by composing the projection $G \to G/N$ with an isomorphism between $G/N$ and $\Z^k$. When $k$ is fixed, we then define a topology on the space of subgroups $N\lhd G$ such that $G/N$ is isomorphic to $\Z^k$. For two such subgroups $N_1$ and $N_2$, we say that $N_1$ is close to $N_2$ if the sphere $S(N_1)$ is contained in a small enough neighbourhood of $S(N_2)$. The space of surjective morphisms $\chi : G \to \Z^k$ is then endowed with the smallest topology making the map $\chi \mapsto {\rm ker}(\chi)$ continuous. Since the BNSR-invariants are open sets, the following proposition is immediate. 

\begin{proposition}\label{prop:openess} Let $G$ be a group of type $\mathscr{F}_{n}$. Let $k\ge 1$ be a natural number. The set of surjective morphisms $\chi : G \to \Z^k$ whose kernel is of type $\mathscr{F}_{n}$ is open. 
\end{proposition}


\section{Construction methods from maps to complex tori}
\label{sec:CxTori}

A common denominator of most of the existing constructions from complex geometry for groups with exotic finiteness properties is that they start from a holomorphic map to a complex torus. We will now describe the two main methods of this kind. The first starts from a holomorphic map with isolated singularities to a one-dimensional torus, while the second requires a finite map to a torus of arbitrary dimension.

\subsection{K\"ahler groups from maps with isolated singularities}

The first construction of K\"ahler groups with arbitrary exotic finiteness properties is due to Dimca, Papadima and Suciu \cite{DimPapSuc-09-II}. Their construction starts from an elliptic curve $E$ and $n\geq 3$ ramified double covers $f_i:S_{g_i}\to E$ where $S_{g_{i}}$ is a Riemann surface of genus $g_i\geq 2$. They show that for the map $f=\sum_{i=1}^{n} f_{i} :S_{g_1}\times \dots \times S_{g_n}\to E$ obtained by summation in $E$, the fundamental group $H=\pi_1(f^{-1}(p))$ of a generic fibre of $f$ is of type $\mathscr{F}_{n-1}$ and not ${\rm FP}_{n}(\Q)$ and is canonically isomorphic to $\ker\left(f_{\ast}: \pi_1(S_{g_1})\times \dots \times \pi_1(S_{g_n})\to \pi_1(E)\cong \mathbb{Z}^2\right)$. This construction has since been extended by Llosa Isenrich \cite{Llo-16-II,Llo-17}, Bridson and Llosa Isenrich \cite{BriLlo-16} and Nicol\'as and Py \cite{NicPy-21}, showing its flexibility.

The main results from~\cite{DimPapSuc-09-II, NicPy-21} can be summarised as follows. 

\begin{theorem}[{Dimca--Papadima--Suciu \cite[Theorem C]{DimPapSuc-09-II}, Nicol\'as--Py \cite[Theorem B]{NicPy-21}}]\label{thm:dpsnp}
\label{thm:NP-DPS}
 Let $M$ be an $n$-dimensional aspherical compact complex manifold with $n\geq 3$, let $S$ be a closed Riemann surface of positive genus and let $f:M\to S$ be a holomorphic map with isolated critical points and connected fibres. Assume that $f$ has at least one critical point. Let $F$ be a smooth generic fibre of $f$. Then $\pi_1(F)$ is of finiteness type $\mathcal{F}_{n-1}$, but not ${\rm FP}_{n}(\Q)$, and is canonically isomorphic to $\mathrm{ker}(f_{\ast}:\pi_1(M)\to \pi_1(S))$.
\end{theorem}

\begin{remark}
 Nicol\'as and Py prove that $\mathrm{ker}(f_{\ast})$ is not of type ${\rm FP}_{n}(\Q)$ using properties of isolated singularities. If $M$ has a nonzero $n$-th $\ell^2$-Betti number and $S$ has genus $1$, then this also follows from \cite[Proposition 14]{LIMP-21}. The nonvanishing of the middle-dimensional $\ell^2$-Betti number occurs for instance if $M$ is K\"ahler hyperbolic with nonzero Euler characteristic \cite{Gro-91,Pan-96}. However, we emphasise that Theorem \ref{thm:NP-DPS} applies in a more general context.
\end{remark}

We observe that in the context of Theorem~\ref{thm:dpsnp}, the kernel of $f_{\ast}$ is a K\"ahler group, being isomorphic to the fundamental group of a generic fibre of $f$. This will not be the case in general for the groups constructed in the next section. We will return to this question in Section~\ref{sec:zeetwo:second}. 

Note that there are generalisations of Theorem \ref{thm:NP-DPS} which relax the conditions that $f$ has isolated singularities and that the image of $f$ is one-dimensional, see \cite[Theorem 2.2]{BriLlo-16} and  \cite[Section 2]{Llo-17}. 

\subsection{Subgroups of K\"ahler groups from finite maps}\label{sec:skgffm}
In \cite{LloPy-22}, the authors of this work presented a second construction method of subgroups of K\"ahler groups with exotic finiteness properties and employed it to show that for every $n\geq 2$ there is a subgroup of a hyperbolic group of type $\mathscr{F}_{n-1}$ but not ${\rm FP}_{n}(\Q)$. This approach is based on ideas of Simpson who studied connectivity properties of sublevel sets of certain harmonic functions $f:\widehat{X}\to \mathbb{R}$ obtained by lifting a harmonic $1$-form to the universal covering of a compact K\"ahler manifold and taking a primitive. See~\cite{Sim-93}, as well as~\cite{LiMaWa-21} for related results. The following theorem summarises the results from \cite{LloPy-22} that we will require here.

\begin{theorem}\label{thm:LP-BNSR}
Let $X$ be a closed aspherical K\"ahler manifold of complex dimension $n\geq 2$. Assume that there exists a finite holomorphic map from $X$ to a complex torus. Then there is a nonzero character $\chi: \pi_1(X)\to \mathbb{Z}$ with kernel of type $\mathscr{F}_{n-1}$. If, moreover, the Euler characteristic of $X$ is non-trivial, then $\ker(\chi)$ is not of type ${\rm FP}_n(\Q)$.
\end{theorem}
\begin{proof}
 This is an immediate consequence of combining \cite[Theorem 6]{LloPy-22}, \cite[Proposition 18]{LloPy-22}, and the openness of the BNSR-invariant.
\end{proof}

Under the assumptions of Theorem \ref{thm:LP-BNSR}, there are in fact many characters with kernels having exotic finiteness properties:

\begin{addendum}\label{add-to-thm:LP-BNSR}
    Under the assumptions of Theorem \ref{thm:LP-BNSR}, $\Sigma^{n-1}(\pi_1(X))\cap  -\Sigma^{n-1}(\pi_1(X))\subseteq S(\pi_{1}(X))$ is a dense open subset. In particular, the set of characters $\chi:\pi_1(X)\to \mathbb{Z}$ with $\ker(\chi)$ of type $\mathcal{F}_{n-1}$ is dense in $S(\pi_{1}(X))$. If, moreover, the Euler characteristic of $X$ is non-zero, then $\ker(\chi)$ is not of type ${\rm FP}_n(\mathbb{Q})$ for any character $\chi$.
\end{addendum}

The first two sentences of the addendum are consequences of~\cite[Theorem 6]{LloPy-22} and~\cite[Proposition 18]{LloPy-22}, together with the properties of the BNSR-invariants recalled in Section~\ref{sec:FinProps}. We now justify the last sentence of the addendum. Under the running assumptions, the manifold $X$ carries a holomorphic $1$-form with finitely many zeros. This follows again from~\cite[Theorem 6]{LloPy-22}. Theorem 10 from~\cite{LloPy-22} implies that $X$ satisfies the conclusion of Singer's conjecture, i.e. the $\ell^{2}$-Betti numbers $b^{(2)}_{i}(X)$ vanish except for $i=n$. We thus have $\chi (X)=(-1)^{n}b_{n}^{(2)}(X)$. Proposition 14 from~\cite{LIMP-21} then yields that the kernel of an arbitrary character of $\pi_{1}(X)$ is not of type ${\rm FP}_{n}(\Q)$ if $X$ has nonzero Euler characteristic.

\begin{remark} Let $X$ be a closed K\"ahler manifold of complex dimension $n$, with finite Albanese map $a_{X} : X\to A(X)$. If $X$ is aspherical, Theorem 10 from~\cite{LloPy-22} implies that $X$ satisfies Singer's conjecture. This can actually be proved without assuming $X$ to be aspherical. Indeed, if $[\omega]$ is a K\"ahler class on $A(X)$, then $a_{X}^{\ast}[\omega]$ is a K\"ahler class on $X$ by~\cite[Prop. 3.6]{GrKi}. Any differential form in the class $[\omega]$ admits, after pull-back to the universal cover of $A(X)$, a primitive of  at most linear growth. Hence the same is true for forms in the class $a_{X}^{\ast}[\omega]$, lifted to the universal cover of $X$. The desired conclusion then follows from~\cite{CaXa}. 
\end{remark}


\section{The Stover--Toledo groups and their subgroups}
\label{sec:ST-Examples}

Stover and Toledo recently constructed, in all dimensions $\geq 2$, smooth complex projective varieties admitting a K\"ahler metric of negative sectional curvature, which are not homotopy equivalent to a locally symmetric manifold, see~\cite{StoTol-21-II}. These are the first such examples in dimension $\geq 4$. Earlier examples had been constructed in dimension $2$ and $3$ by Mostow--Siu, Deraux and Zheng~\cite{MoSi-80, zheng1, zheng2, Deraux-04, Deraux-05}. Stover and Toledo's examples are obtained by ramified cover of suitable congruence covers of arithmetic ball quotients of the simplest type. We summarise here their work and then apply the construction described in Section~\ref{sec:skgffm} to the fundamental groups of the corresponding negatively curved K\"ahler manifolds, yielding a proof of Theorem~\ref{thm:ST-Intro}.

\subsection{Complex ball quotients and the Stover--Toledo construction}\label{subsec:stto}

For $m\geq 1$ we denote  by ${\rm PU}(m,1)$ the group of holomorphic isometries of the unit ball $\mathbb{B}^m_{\C}$ of $\C^m$ equipped with the Bergman metric. We will consider cocompact lattices $\Gamma< {\rm PU}(m,1)$ which are \emph{arithmetic}. We refer the reader to~\cite{Mar-91, Zim} for the definition of this notion. More specifically,  we will be interested in uniform arithmetic lattices of the \emph{simplest type}, whose definition we now recall. 

Let $F\subset \R$ be a totally real number field, let $E\subset \C$ be a purely imaginary quadratic extension of $F$ and let $V=E^{n+1}$. Assume that we are given a Hermitian form $H: V\times V \to E$ such that its extension to $V\otimes \C$ has signature $(m,1)$. 

Given an embedding $\sigma: E\to \C$ we denote by $H^{\sigma}$ the Hermitian form on $\C^{n+1}$ obtained by applying $\sigma$ to the coefficients of the matrix representing $H$ in the canonical basis of $V$. Assume that for every embedding $\sigma: E\to \C$ with $\sigma|_F\neq {\rm id}|_F$, the twisted Hermitian form $H^{\sigma}$ has signature $(m+1,0)$. Let $\mathcal{O}_E$ be the ring of integers of $E$ and let $U(H,\mathcal{O}_E)$ be the group of $(m+1)\times (m+1)$-matrices with coefficients in $\mathcal{O}_E$ which preserve $H$. Then $U(H,\mathcal{O}_E)$ is a lattice in the group $U(V\otimes \C,H)$ of automorphisms of $(V\otimes \C, H)$.  It is uniform if and only if $F\neq \Q$. We call a lattice $\Gamma < {\rm PU}(n,1)$ of the \emph{simplest type} if it is commensurable to a lattice of the form $U(H,\mathcal{O}_E)$.

For $m\geq 2$, let $X=\Gamma \backslash \mathds{B}^m_{\C}$ be a smooth compact complex hyperbolic $m$-manifold. We call a pair $(X,D)$ a \emph{good pair} if $D=D_1\cup D_2\cup \dots \cup D_k\subset X$ is a non-trivial divisor such that the $D_i$ are pairwise non-intersecting smooth codimension one subvarieties of $X$. We call $D$ \emph{totally geodesic} if the embeddings of the $D_i$ in $X$ are totally geodesic. 

 If $\Gamma$ is of the simplest type, then $X$ admits totally geodesically immersed divisors. Up to passing to a finite cover of $X$, we can find such divisors which are embedded. We will require the following more precise version of this result (see~\cite[Section 5]{StoTol-21-II}).
\begin{lemma}\label{lem:totgeod}
 Let $\Gamma<{\rm PU}(m,1)$ be a torsion-free congruence arithmetic lattice of the simplest type and let $X=\Gamma \backslash \mathbb{B}^m_{\C}$. Then there exists a finite congruence cover $p:X'\to X$ and a divisor $D'\subset X'$ such that $(X',D')$ is a totally geodesic good pair. 
\end{lemma}
Note that conversely, if $m\geq 2$ and $(X,D)$ is a totally geodesic good pair, with $X$ arithmetic, then $\Gamma$ is of the simplest type by \cite[Proposition 3.2]{StoTol-21-II}.

\begin{theorem}[{\cite[Theorem 1.5 and Proposition 5.1]{StoTol-21-II}}]\label{thm:StoverToledo}
Assume that $\Gamma < {\rm PU}(m,1)$ is a cocompact torsion-free congruence arithmetic lattice of the simplest type and let $X=\Gamma\backslash \mathbb{B}^m_{\mathbb{C}}$. Let $D\subset X$ be a divisor such that $(X,D)$ is a totally geodesic good pair. Let $d\geq 2$. Then there is a finite cover $p:X'\to X$ which admits a cyclic $d$-fold ramified cover $Y\to X'$ with ramification locus the totally geodesic divisor $D':=p^{-1}(D)$. The cover $Y$ is a smooth projective variety which admits a metric of negative sectional curvature and is not homotopy equivalent to any locally symmetric manifold.
\end{theorem}

Let us make a few comments on Theorem~\ref{thm:StoverToledo}. The fact that ramified covers of ball quotients along totally geodesic divisors admit negatively curved K\"ahler metrics was known prior to~\cite{StoTol-21-II} and is due to Zheng~\cite{zheng1}, who generalized earlier work by Mostow--Siu~\cite{MoSi-80}. The key contribution made in~\cite{StoTol-21-II} is to show that one can find many arithmetic ball quotients $X$ containing totally geodesic divisors $D$, forming a good pair and such that the integral homology class of $D$ is divisible by some nontrivial integer, thus allowing to build cyclic ramified covers. The new contribution (the divisibility of the homology class $[D]$) relies on deep results on the cohomology of arithmetic groups due to Bergeron, Millson and Moeglin~\cite{BeMiMo-16}. 

\begin{remark} Arithmeticity of the lattices under consideration appears in two ways in this work. Firstly, through the results on the cohomology of arithmetic groups used in~\cite{StoTol-21-II}, and secondly through properties of the Albanese map of arithmetic ball quotients~\cite{Eys-18, LloPy-22}. In this second appearance, arithmeticity is used in a much more elementary way. 
\end{remark}

\subsection{New subgroups of hyperbolic groups of type $\mathscr{F}_{m-1}$ and not $\mathscr{F}_m$}

We will check that we can apply Theorem \ref{thm:LP-BNSR} and Addendum \ref{add-to-thm:LP-BNSR} to the Stover--Toledo examples. To do so, we first need the following:

\begin{lemma}\label{lem:euler}
 Every compact K\"ahler manifold $Y$ as in Theorem \ref{thm:StoverToledo} has non-zero Euler characteristic.
\end{lemma}

\begin{proof} Recall that the Euler chracteristic is multiplicative under finite etale covers. For ramified covers, there is a similar formula, taking into account the ramification. It reads as follows, the notation being as in Theorem~\ref{thm:StoverToledo}:
$$\chi (Y)-\chi (D')=d(\chi (X')-\chi (D)).$$ 
Since $D$ and $D'$ are diffeomorphic, we obtain:
\begin{equation}\label{eq:eulerch}
\chi (Y)=d\chi (X')+(1-d)\chi(D).
\end{equation}
We now use the fact that compact ball quotients of dimension $k$ have nonzero Euler characteristic, of the same sign as $(-1)^k$. Hence, $\chi (X')$ and $\chi (D)$ are both nonzero, of opposite signs. Thus the two terms on the right-hand side of~\eqref{eq:eulerch} are nonzero of the same sign and $\chi (Y)\neq 0$. Alternatively, we could have appealed to Gromov's work~\cite{Gro-91} to justify the nonvanishing of $\chi (Y)$, but the above argument is simpler.\end{proof}

We are now ready to state and prove the main result of this section. Theorem \ref{thm:ST-Intro} is a direct consequence.
\begin{theorem}\label{thm:ST-main-text}
 Let $m\geq 2$, let $\Gamma<{\rm PU}(m,1)$ be a uniform torsion-free congruence arithmetic lattice of the simplest type and let $X=\Gamma \backslash \mathbb{B}^m_{\C}$. Let $d\ge 2$. Then there is a finite cover $X'\to X$ such that there exists a finite ramified cover $Y'_d\to X'$ with the following properties:
 \begin{enumerate}
     \item $Y'_d$ is a smooth projective variety which admits a metric of negative sectional curvature and is not homotopy equivalent to any locally symmetric manifold,
     \item the $(m-1)$-th BNSR-invariant $\Sigma^{m-1}(\pi_1(Y'_d))$ is dense in the character sphere $S(\pi_1(Y'_d))$ and every rational character $\xi\in {\rm Hom}(\pi_1(Y'_d),\Q)-\{0\}$ such that $[\xi]$ is in the dense open set $\Sigma^{m-1}(\pi_1(Y'_d))\cap -\Sigma^{m-1}(\pi_1(Y'_d))$ satisfies that $\ker(\xi)$ is of type $\mathscr{F}_{m-1}$ but not of type ${\rm FP}_m(\Q)$.
 \end{enumerate}
\end{theorem}

\begin{proof}
 By Lemma \ref{lem:totgeod}, there exists a finite congruence cover $X_1\to X$ and a divisor $D_1\subset X_1$ such that $(X_1,D_1)$ is a totally geodesic good pair. Theorem \ref{thm:StoverToledo} implies that there is a finite congruence cover $p_2: X_2\to X_1$ which admits a $d$-fold cyclic ramified cover $q: Y_d\to X_2$ with ramification locus the totally geodesic divisor $D_2=p_2^{-1}(D_1)$. The manifold $Y_d$ admits a K\"ahler metric of negative sectional curvature and does not have the homotopy type of a locally symmetric space. By a theorem of Eyssidieux \cite{Eys-18} (see also \cite[Theorem 24]{LloPy-22}) there is a further finite cover $p_3:X_3\to X_2$ such that the Albanese map $a_{X_3}\colon X_3\to A(X_3)$ is an immersion and thus defines a finite map to a complex torus. The pair $(X_3,D_3=p_3^{-1}(D_2))$ is again a totally geodesic good pair. Moreover, $p_3$ induces a regular cover $Y'_d\to Y_d$ such that there is a $d$-fold ramified cover $Y'_d\to X_3$ with ramification locus $D_3$. In particular, $Y'_d$ is a smooth projective variety admitting a metric of negative sectional curvature. Since finite (possibly ramified) covers are finite maps and compositions of finite maps are finite, the induced holomorphic map $Y'_d\to A(X_3)$ is finite. By Lemma \ref{lem:euler}, $Y'_d$ has non-trivial Euler characteristic. Addendum \ref{add-to-thm:LP-BNSR} thus completes the proof. 
\end{proof}

\subsection{New K\"ahler groups of type $\mathscr{F}_{2n-1}$ and not $\mathscr{F}_{2n}$}\label{sec:ST-CS}

We sketch here without details a possible construction of groups as in the title of this section. Let $X_{{\rm CS}}$ be the {\it Cartwright-Steger surface}. It is a quotient of the unit ball of $\C^2$ by a uniform congruence arithmetic lattice $\Gamma \leq {\rm PU}(2,1)$ of the simplest type \cite[p. 89-90]{Sto-14} (see also \cite[Remark 3.7]{StoTol-21-II}) with $b_1(X)=2$. Consequently its Albanese map $f:X_{{\rm CS}}\to E$ is onto an elliptic curve $E$. It was shown in \cite{CarKozYeu-17} that this map has isolated singularities. Moreover, $f$ has connected fibres. Let $n\ge 2$ be an integer. Consider the map
$$F : X_{{\rm CS}} \times \cdots \times X_{{\rm CS}}\to E$$
from the direct product of $n$ copies of $X_{{\rm CS}}$ to $E$ obtained by summing the map $f$ appplied to each factor. It was proved in~\cite{NicPy-21} that the fundamental group of the generic fibre of the map $F$ is a K\"ahler group of type $\mathscr{F}_{2n-1}$ but not of type ${\rm FP}_{2n}(\mathbb{Q})$. Applying Stover and Toledo's work, we can consider a congruence cover $p : X' \to X_{{\rm CS}}$ admitting a ramified cover $p' : Y\to X'$ along a totally geodesic divisor. If the ramification locus in $X'$ is chosen in general position, it will project {\it onto} $E$ via the map $f\circ p$. This implies that the map
$$h:=f\circ p\circ p' : Y\to E$$
has isolated critical points. One can thus repeat\footnote{One must possibly replace $E$ by a finite cover to ensure that $h$ has connected fibers.} the construction from~\cite{NicPy-21} by considering the map
$$h+\cdots +h : Y \times \cdots \times Y\to E$$
from the direct product of $n\ge 2$ copies of $Y$. This yields new K\"ahler groups with exotic finiteness properties, sitting inside a direct product of hyperbolic groups.


\section{Homomorphisms to $\Z^{2}$}\label{sec:zeetwo}

In this section, we prove Theorem~\ref{thm:zeetwo} (Section~\ref{sec:zeetwo:first}) and then discuss whether the groups with exotic finiteness properties constructed in this article and in~\cite{LloPy-22} are K\"ahler or not (Section~\ref{sec:zeetwo:second}). 
 
\subsection{Hodge theory yields circles in the BNSR-invariants}\label{sec:zeetwo:first}

Let $X$ be a closed K\"ahler manifold of complex dimension $n$. Let $a\in H^{1}(X,\C)$ be a cohomology class whose real and imaginary part are independent in $H^{1}(X,\R)$. Equivalently, we require that $a$ is not a complex multiple of a real class. Associated to $a$, there is a circle 
$$S^{1}(a)\subset S(\pi_{1}(X))\cong H^{1}(X,\R)-\{0\}/\R_{+}^{\ast}.$$
The circle $S^1(a)$ is made of the projections in $S(\pi_{1}(X))$ of the (de Rham) cohomology classes
$${\rm Re}(e^{i\theta}a) \;\;\; (\theta \in \R).$$
We now assume that $a$ is the cohomology class of a holomorphic $1$-form $\alpha$ with finitely many zeros. In particular, if $X$ is aspherical, Theorem 6 from~\cite{LloPy-22} can be applied to the cohomology classes $[{\rm Re}(e^{i\theta}\alpha)]$ for each real number $\theta$, since $e^{i\theta}\alpha$ is a holomorphic $1$-form with finitely many zeros. That theorem gives that $[{\rm Re}(e^{i\theta}\alpha)]\in \Sigma^{n-1}(\pi_{1}(X))\cap -\Sigma^{n-1}(\pi_{1}(X))$. In other words, the circle $S^{1}([\alpha])$ is contained in $\Sigma^{n-1}(\pi_{1}(X))\cap -\Sigma^{n-1}(\pi_{1}(X))$. We have proved:

\begin{proposition}\label{prop:cibnsr} If $X$ is aspherical and if $\alpha$ is a holomorphic $1$-form with finitely many zeros on $X$ we have:
$$S^{1}([\alpha])\subset \Sigma^{n-1}(\pi_{1}(X))\cap -\Sigma^{n-1}(\pi_{1}(X)).$$
\end{proposition}

We are now ready to prove Theorem~\ref{thm:zeetwo}, combining Proposition~\ref{prop:cibnsr} with the results recalled in Section~\ref{sec:FinProps}. 

\begin{proof}[Proof of Theorem~\ref{thm:zeetwo}] Let $X$ be as in the statement of Theorem~\ref{thm:zeetwo}. The set $U\subset H^{0}(X,\Omega^{1}_{X})$ of holomorphic $1$-forms with finitely many zeros is dense (see Propositions 14 and 18 in~\cite{LloPy-22}). Let $\alpha \in U$. By Proposition~\ref{prop:cibnsr}, the circle $S^{1}([\alpha])$ is contained in $\Sigma^{n-1}(\pi_{1}(X))\cap -\Sigma^{n-1}(\pi_{1}(X))$. Let $a_{1}$ and $a_2$ be rational elements in $H^{1}(X,\R)$ which are close enough to $[{\rm Re}(\alpha)]$ and $[{\rm Im}(\alpha)]$ respectively. The set $\Sigma^{n-1}(\pi_{1}(X))\cap -\Sigma^{n-1}(\pi_{1}(X))$ being open, we have
\begin{equation}\label{eq:fdetr}
[\cos (t)a_1 +\sin (t) a_2] \in \Sigma^{n-1}(\pi_{1}(X))\cap -\Sigma^{n-1}(\pi_{1}(X))
\end{equation} for all real numbers $t$, if $a_{1}$ and $a_{2}$ are close enough to $[{\rm Re}(\alpha)]$ and $[{\rm Im}(\alpha)]$. Let $N$ be a large enough integer so that $Na_{1}$ and $Na_{2}$ define morphisms from $\pi_{1}(X)$ to $\mathbb{Z}$. The image $\Lambda$ of the morphism $(Na_1,Na_2) : \pi_{1}(X)\to \Z^2$ is isomorphic to $\Z^2$. Equation~\eqref{eq:fdetr} implies that for any nonzero morphism $u : \Lambda \to \R$, we have $[u\circ (Na_1,Na_2)]\in \Sigma^{n-1}(\pi_{1}(X))\cap -\Sigma^{n-1}(\pi_{1}(X))$. Theorem~\ref{thm:finiteness-high-rank} then implies that the kernel of the morphism
$$(Na_1,Na_2) : \pi_{1}(X)\to \Lambda \cong \Z^{2}$$
is of type $\mathscr{F}_{n-1}$. The fact that it is not of type ${\rm FP}_n(\mathbb{Q})$ follows from the last point of Addendum~\ref{add-to-thm:LP-BNSR}; indeed, the latter property holds for kernels of arbitrary characters. Finally the openness of the set of homomorphisms $\pi_{1}(X)\to \Z^{2}$ whose kernel is of type $\mathscr{F}_{n-1}$ follows from Proposition~\ref{prop:openess}. This concludes the proof of Theorem~\ref{thm:zeetwo}. 
\end{proof}

Theorem \ref{thm:zeetwo} has the following consequence, which also proves Corollary \ref{cor:zeetwo}.

\begin{corollary}\label{cor:zeetwo-main-text}
    Assume that $Z$ is a $n$-dimensional complex manifold which is either a complex ball quotient by a uniform arithmetic lattice with non-trivial first Betti number, or one of the Stover--Toledo manifolds $Y_d'$ constructed in Theorem \ref{thm:ST-main-text}. Then there is a finite cover $Z'\to Z$ and a surjective homomorphism $\phi:\pi_1(Z')\to \mathbb{Z}^2$ with kernel of type $\mathscr{F}_{n-1}$, but not ${\rm FP}_n(\mathbb{Q})$.
\end{corollary}
\begin{proof}
    We argue as in the proof of Theorem \ref{thm:ST-main-text}, using Eyssidieux's work \cite{Eys-18}, that in both cases there is a finite cover $Z'\to Z$ with finite Albanese map and non-trivial Euler characteristic. The result is then an immediate consequence of Theorem \ref{thm:zeetwo}.
\end{proof}

Since the virtual first Betti number of arithmetic lattices $\Gamma\le {\rm PU}(n,1)$ with $b_1(\Gamma)>0$ is infinite \cite{Ago-06, Ven-08}, it is natural to ask if Corollary \ref{cor:zeetwo-main-text} generalises to abelian quotients of arbitrary rank.

\begin{question}
    Let $n\geq 2$ and $k\geq 3$ be integers. Let $Z$ be as in Corollary \ref{cor:zeetwo-main-text} with ${\rm dim}_{\mathbb{C}}(Z)=n$. Is there a finite cover $Z'\to Z$ whose fundamental group admits a surjective homomorphism $\phi:\pi_1(Z')\to \mathbb{Z}^k$ with kernel of type $\mathscr{F}_{n-1}$ and not of type ${\rm FP}_n(\mathbb{Q})$?
\end{question}

A positive answer to this question would also provide a positive answer to the following more general question.

\begin{question}
 Let $n\geq 4$ and $k\geq 3$ be integers. Is there a hyperbolic group $G$ together with a surjective homomorphism $\phi: G\to \mathbb{Z}^k$ such that $\ker(\phi)$ is of type $\mathscr{F}_{n-1}$ and not of type ${\rm FP}_n(\mathbb{Q})$?
\end{question}
For $n=2$, this last question can be answered using a classical construction due to Rips, see~\cite[p. 227]{BrHa} and very recently Kropholler and Llosa Isenrich gave an answer for $n=3$ \cite{Llo-Kro}. The case $n>3$ remains open.

\subsection{K\"ahler and non-K\"ahler subgroups}\label{sec:zeetwo:second}

We mentioned in the introduction that complex Morse theory methods allow to construct examples of groups with exotic finiteness properties as subgroups of K\"ahler groups. These groups are sometimes K\"ahler but not always. In the next theorem, we justify that some of these groups are not K\"ahler. We focus on subgroups of lattices in ${\rm PU}(n,1)$ and explain in Remark~\ref{rem:kcnpst} how one could possibly obtain slightly more general results. 

\begin{theorem}\label{thm:nonk} Let $\Gamma < {\rm PU}(n,1)$ be a torsion-free cocompact lattice and let $\psi : \Gamma \to \Z^{k}$ be a surjective homomorphism. 
\begin{itemize}
\item If $k=1$, the kernel of $\psi$ is not a K\"ahler group. 
\item If $k=2$, and if the kernel of $\psi$ is a K\"ahler group, there exists a holomorphic map $\pi$ with connected fibres from the quotient $\Gamma \backslash\mathbb{B}_{\C}^{n}$ onto an elliptic curve $E$, and an isomorphism $\varphi : \pi_{1}(E)\to \Z^2$ such that $\psi=\varphi\circ \pi_{\ast}$.  
\end{itemize} 
\end{theorem}

When $n\le 2$, the kernel of a morphism $\psi$ as above cannot be finitely presented. This follows from Corollary 15 in~\cite{LIMP-21}; see also~\cite{fisher}. Hence Theorem~\ref{thm:nonk} is relevant only for $n\ge 3$. As a consequence of it, we obtain that in the context of Theorem~\ref{thm:zeetwo} applied to ball quotients, most of the kernels under consideration are not K\"ahler. Indeed, let $X$ be a closed ball quotient with finite Albanese map. Let $\psi : \pi_{1}(X)\to \Z^2$ be a surjective morphism. The space $H^{1}(X,\R)$ is endowed with its natural complex structure via the identification
$$H^{1,0}(X)\to H^{1}(X,\R)$$
mapping a holomorphic $1$-form to the class of its real part. If the kernel of $\psi$ is K\"ahler, then the plane $P\subset H^{1}(X,\R)$ defined by $\psi$ is a complex line, thanks to Theorem~\ref{thm:nonk}. This last condition can be broken by a suitable rational perturbation of $\psi$, yielding many morphisms onto $\Z^2$ with non-K\"ahler kernels. Note that we assumed that the Albanese map of $X$ was finite, hence $b_{1}(X)>2$ and $H^{1}(X,\Q)$ is large enough to perturb $\psi$.

\begin{proof}[Proof of Theorem~\ref{thm:nonk}] As observed above, we can assume that $n\ge 2$. We fix $\Gamma$ and $\psi$ as in the statement of the theorem. We write $X=\Gamma \backslash \mathbb{B}_{\mathbb{C}}^{n}$. We assume that there exists a closed K\"ahler manifold $Y$ whose fundamental group is isomorphic to the kernel of $\psi$. The natural morphism $\varrho : \pi_{1}(Y) \to \pi_{1}(X)$ is induced by a smooth map $h : Y\to X$. By work of Eells and Sampson there exists a harmonic map $h_0 : Y \to X$ homotopic to $h$, see~\cite{EelSam-64} and also~\cite[p. 68]{ABCKT-95}. This map is pluriharmonic and the $(1,0)$ part of its differential  
$$dh_{0}^{1,0} : TY \to TX\otimes \C$$
is holomorphic for a suitable holomorphic structure on the bundle $TX\otimes \C$, see~\cite[Ch. 6]{ABCKT-95} or~\cite[\S 9.2.2]{Py-book}. We distinguish two cases.

If the complex rank of $dh_{0}^{1,0}$ is equal to $1$, then the harmonic map $h_0$ factors through a Riemann surface: there exists a surjective holomorphic map with connected fibres onto a Riemann surface, denoted by $\pi : Y \to \Sigma$, and a harmonic map $m_{0} : \Sigma \to X$ such that 
$$h_0=m_{0}\circ \pi.$$
This result is due to Carlson and Toledo~\cite[\S 7]{CarTol-89}. The kernel of the induced morphism $\pi_{\ast}$ is contained in the kernel of $\varrho$. Since $\varrho$ is injective, so is $\pi_{\ast}$. Since $\pi$ has connected fibres, $\pi_{\ast}$ is also surjective, and we obtain that
$$\pi_{\ast} : \pi_{1}(Y)\to \pi_{1}(\Sigma)$$
is an isomorphism. However, $\pi_1(Y)$ is not of type ${\rm FP}_n(\mathbb{Q})$ by Addendum \ref{add-to-thm:LP-BNSR}, while surface groups are of type $\mathscr{F}_r$ for all $r$, yielding a contradiction. Alternatively, one can also observe that $\Gamma$ fits into an extension
$$1 \to \pi_{1}(\Sigma)\to \Gamma \to \Z^k \to 1$$
to obtain a contradiction by a theorem due to Bregman and Zhang \cite{BrZh} (see also the work of Nicol\'as \cite{Nic-22} for a different proof of this theorem). This first half of the proof works for arbitrary $k$.

We now assume that the complex rank of $dh_{0}^{1,0}$ is greater than $1$. In that case another result due to Carlson and Toledo~\cite[Cor. 3.7]{CarTol-89} ensures that the map $h_0$ is holomorphic (after possibly replacing the complex structure on $X$ by its complex conjugate structure). To complete the proof we use arguments similar to those in \cite[\S 3]{Llo-18-I} and \cite[p.19--20]{Llo-17}. The kernel of the pull-back map 
$$h_{0}^{\ast} : H^{1}(X,\C)\to H^{1}(Y,\C)$$
is a sub-Hodge structure of $H^{1}(X,\C)$. In particular it has even dimension. By the original definition of $\pi_{1}(Y)$, which is isomorphic to the kernel of $\psi$, we obtain a contradiction if $k=1$. This concludes the proof that ${\rm ker}(\psi)$ is not K\"ahler for $k=1$. When $k=2$, we consider the Albanese tori of $X$ and $Y$, denoted by $A(X)$ and $A(Y)$. We let $a_X : X \to A(X)$ and $a_Y : Y \to A(Y)$ be the corresponding Albanese maps. By the universal property of Albanese maps, there exists a holomorphic map $\theta : A(Y)\to A(X)$ making the following diagram commutative:
\[\xymatrix{Y \ar[r]^{h_{0}} \ar[d]^{a_{Y}} & X \ar[d]^{a_{X}} \\
A(Y) \ar[r]^{\theta} & A(X). \\
}\]
We can assume that $\theta$ is a group morphism, after possibly composing $a_X$ by a translation. The codimension of the image of $A(Y)$ in $A(X)$ is equal to the dimension of the kernel of the restriction map $H^{1,0}(X)\to H^{1,0}(Y)$, which is equal to $\frac{k}{2}=1$. Hence $\theta(A(Y))$ is a codimension $1$ subtorus of $A(X)$. We let $E$ be the quotient $A(X)/\theta (A(Y))$ and $\pi$ be the composition of the map $a_X$ with the quotient map $A(X)\to E$. Let $a\in H^{1}(X,\Z)$. We write $a=a_X^{\ast}b$ with $b\in H^{1}(A(X),\Z)$. Then $a$ vanishes on $\pi_{1}(Y)$ if and only if $b$ vanishes on the image of $\pi_{1}(A(Y))$ in $\pi_{1}(A(X))$. This is also equivalent to the fact that $b$ comes from a class in $H^{1}(E,\Z)$. This implies that $\psi=\varphi \circ \pi_{\ast}$ for some morphism $\varphi : \pi_{1}(E)\to \Z^2$ which is necessarily an isomorphism. We finally have to justify the fact that $\pi$ has connected fibers. If this is not the case, we can consider the Stein factorization $X\to \Sigma \to E$ of this map. If $\Sigma\neq E$ and $\Sigma$ is hyperbolic, we obtain a contradiction since the kernel of the morphism $\pi_{1}(\Sigma)\to \pi_{1}(E)$ is not finitely generated. If $\Sigma$ has genus one, it cannot be a nontrivial covering space of $E$, otherwise $\pi_{\ast}$ would not be surjective. Hence $\Sigma =E$ and $\pi$ has connected fibers. This concludes the proof of Theorem~\ref{thm:nonk}.\end{proof}

\begin{remark}\label{rem:kcnpst} One could try to prove a statement analogous to Theorem~\ref{thm:nonk}, when the lattice $\Gamma < {\rm PU}(n,1)$ is replaced by the fundamental group of one of the negatively curved manifolds built by Stover and Toledo and described in Section~\ref{subsec:stto}. This can most probably be done along the following lines. The natural K\"ahler form on these manifolds (see~\cite{MoSi-80, zheng1}) is built by pulling back the K\"ahler form of the ball quotient and adding a suitable $(1,1)$ form supported near the branch locus. This metric is known to have nonpositive Hermitian sectional curvature (see~\cite[Ch. 6]{ABCKT-95} or~\cite[\S 9.2.2]{Py-book} for this notion). This property is enough to build a harmonic map and to prove that it is pluriharmonic, as in the proof of Theorem~\ref{thm:nonk}. Then, one should adapt the results from~\cite{CarTol-89}, about harmonic maps to complex hyperbolic manifolds, to the case where the target is one of the manifolds built by Stover--Toledo.   
\end{remark}


\section{Fibre products over Riemann surfaces}
\label{sec:Livne}

In this section we describe a new method that uses complex Morse theory to produce examples of non-normal subgroups with exotic finiteness properties in a direct product of non-positively curved groups. Most previous constructions of subgroups with exotic finiteness properties in non-positively curved groups lead to normal subgroups. Our method might possibly be applied in other situations and could lead to further interesting examples.

\subsection{Fibre products with isolated singularities and Livn\'e's surfaces}
\label{sec:Liv-Ex}

Assume that $X_1$ and $X_2$ are closed complex manifolds of dimensions $n_i=\dim_{\C}(X_i)$ with $n_1+n_2\geq 3$ and that $p_i: X_i\to \Sigma$ ($i=1, 2$), is a surjective holomorphic map with isolated critical points onto a closed hyperbolic Riemann surface $\Sigma$. Assume further that $p_1$ and $p_2$ induce surjections on fundamental groups, that the fibres of $p_1$ are connected, and that the sets of critical values of $p_1$ and $p_2$ are both non-empty and have empty intersection. Let
\[
    Z:=\left\{(x_1,x_2)\in X_1\times X_2 \mid p_1(x_1)=p_2(x_2)\right\}\subset X_1\times X_2
\]
be the fibre product and observe that $Z=(p_1,p_2)^{-1}(\Delta_{\Sigma})$ for $\Delta_{\Sigma}\hookrightarrow \Sigma \times \Sigma$ the diagonal. Our hypothesis on critical values implies that $Z$ is a smooth connected complex submanifold of $X_1\times X_2$.

Denote by $q: Y\to \Sigma\times \Sigma$ the covering space corresponding to the subgroup $\pi_{1}(\Delta_{\Sigma})<\pi_{1}(\Sigma)\times \pi_{1}(\Sigma)$ and let $\widehat{q}:W\to X_1\times X_2$ be the covering space induced by $(p_1,p_2)_{\ast}^{-1}(\pi_1(\Delta_{\Sigma}))$. We fix a hyperbolic metric on $\Sigma$ and equip $Y$ with the metric of non-positive curvature obtained by pulling back the corresponding product metric on $\Sigma \times \Sigma$. Let $\widehat{\Delta}_{\Sigma}$ be the unique {\it compact} connected component of the preimage of $\Delta_{\Sigma}$ in $Y$. Let
$$
\begin{array}{cccc}
f: & Y & \to & \mathbb{R}_{\geq 0} \\
& y & \mapsto & \left({\rm dist}\left(y,\widehat{\Delta}_{\Sigma}\right)\right)^2 \\
\end{array}
$$
be the square of the distance function to $\widehat{\Delta}_{\Sigma}$ in $Y$. Since $\Sigma$ is non-positively curved, for every $y\in Y$ there is a unique shortest geodesic connecting $y$ to a point in $\widehat{\Delta}_{\Sigma}$ (it intersects $\widehat{\Delta}_{\Sigma}$ orthogonally). This implies that the restriction of $f$ to $f^{-1}((0,\infty))= Y \setminus \widehat{\Delta}_{\Sigma}$ is a submersion. One can in fact show that $Y$ is diffeomorphic to the normal bundle of $\widehat{\Delta}_{\Sigma}$, see~\cite[Section 2]{Sik-87}. We fix once and for all a map $\pi : W \to Y$ lifting the composition $(p_{1},p_{2})\circ \widehat{q}$, so that we have a commutative diagram: 
\[
    \xymatrix{
    & W\ar[r]^{\widehat{q}}\ar[ld]_{f\circ \pi}\ar[d]^{\pi}& X_1\times X_2\ar[d]^{(p_1,p_2)}\\
    \mathbb{R}_{\geq 0}& Y \ar[l]^{f}\ar[r]_{q} &\Sigma \times \Sigma.
    }
\]
Observe that the maps $f$ and $\pi$ are proper. Thus, the same applies to $f\circ\pi$. Since $p_1$ and $p_2$ induce surjections on fundamental groups, the fibre product of $Y$ and $X_1 \times X_2$ is connected, hence can be identified with $W$. We then define:
$$\widehat{Z}:=\pi^{-1}(\widehat{\Delta}_{\Sigma}).$$
The description of $W$ as a fibre product implies that $\widehat{q}$ induces a diffeomorphism between $\widehat{Z}$ and $Z$.

In the next two sections, we shall prove the following theorem:

\begin{theorem}\label{thm:FibreProduct}
The manifold $W$ and the group $\pi_{1}(Z)$ enjoy the following properties:
    \begin{enumerate}
        \item the inclusion $\widehat{Z} \hookrightarrow W$ induces an isomorphism on homotopy groups up to degree $n_{1}+n_{2}-2$, in particular $W$ and $Z$ have isomorphic fundamental groups,
        \item $H_{n_1+n_2}(W,\mathbb{Q})$ is infinite dimensional,
        \item $\pi_1(Z)$ is the group theoretic fibre product of the surjective homomorphisms $p_{1\ast}:\pi_1(X_{1})\to \pi_1(\Sigma)$ and $p_{2\ast}:\pi_1(X_{2})\to \pi_1(\Sigma)$,
        \item if $X_1$ and $X_2$ are aspherical, then $\pi_1(Z)$ is of finiteness type $\mathscr{F}_{n_1+n_2-1}$ but not of type ${\rm FP}_{n_1+n_2}(\mathbb{Q})$.
            \end{enumerate}
\end{theorem}

Theorem~\ref{thm:Livne} from the introduction follows immediately from Theorem~\ref{thm:FibreProduct}. Let us now turn to the examples. In \cite{Liv-81}, Livn\'e constructed a family of smooth compact complex algebraic surfaces $S_d(N)$ by taking ramified covers of the compactification $E(N)$ of the universal elliptic curve with full level $N$ structure for every integer $N\geq 5$. The holomorphic map $E(N)\to X(N)$ to the compactification of the moduli space of elliptic curves with a full level $N$ structure has isolated singularities. Since the ramified covering is along a divisor whose irreducible components are sections of $E(N)\to X(N)$, the induced holomorphic map $S_d(N)\to X(N)$ also has isolated singularities. Livn\'e proves that for $(N,d)\in \left\{(7,7), (8,4), (9,3), (12,2)\right\}$ the surface $S_d(N)$ is a complex  ball quotient. Since $X(N)$ is a closed hyperbolic Riemann surface, this provides examples of 2-dimensional complex ball quotients admitting maps with isolated singularities to closed hyperbolic Riemann surfaces, to which one can apply Theorem~\ref{thm:Livne}. 

The proof of Theorem~\ref{thm:FibreProduct} is given in the next two subsections.

\subsection{A perturbation of the distance function}\label{sec:perturbation}
In this section, by perturbing the function $f\circ \pi$, we will construct a proper function 
$$h : W \to \R_{\ge 0}$$
which is Morse outside of $\widehat{Z}$ and all of whose critical points outside of $\widehat{Z}$ have index equal to $n_{1}+n_{2}$. This function will be used in Section~\ref{sec:ComplexMorse} to prove Thereom~\ref{thm:FibreProduct}. 

We start by observing that $f\circ \pi|_{W\setminus \widehat{Z}}$ has isolated critical points, coinciding with the zeros of the differential of $\pi$. In the lemma below, we denote by $d_{w}\varphi$ the differential of a smooth function $\varphi$ at a point $w$. 

\begin{lemma}\label{lem:crdi}
\label{lem:Sing}
    Fix a point $w\in W\setminus \widehat{Z}$. The following are equivalent:
    \begin{enumerate}
        \item ${\rm d}_{w}(f\circ\pi)=0$
        \item ${\rm d}_w \pi=0$
        \item ${\rm d}_{\widehat{q}(w)}(p_1,p_2)=0$.
    \end{enumerate}
If a function $f' : W \to \R_{\ge 0}$ is sufficiently $C^{1}$-close to $f$ near $w$, the conditions above are also equivalent to $d_{w}(f'\circ \pi)=0$.      
\end{lemma}
\begin{proof}
The equivalence of (2) and (3) is a direct consequence of the fact that $q$ and $\widehat{q}$ are covering maps. By the chain rule, (2) implies (1). We continue the proof by showing that (1) implies (2). 

Recall that the restriction of $f$ to $Y\setminus \widehat{\Delta}_{\Sigma}$ is a submersion. Let $y\in Y\setminus \widehat{\Delta}_{\Sigma}$ and let $$(s_1,s_2):=q(y)\in \Sigma\times \Sigma.$$ The orthogonal decomposition $$T_{(s_1,s_2)}\Sigma\times\Sigma = T_{s_1}\Sigma \oplus T_{s_2}\Sigma$$ pulls back to a decomposition $T_yY= E_{y,1}\oplus E_{y,2}$, with $E_{y,i}=({\rm d}_{y}q)^{-1}(T_{s_i}\Sigma)$. Let $\mathbb{H}$ be the hyperbolic plane, thought of as the universal cover of $\Sigma$. The orthogonal projection of a point $(a,b)\in \mathbb{H}\times \mathbb{H}$ on the diagonal $\Delta_{\mathbb{H}}$ is the point $(z,z)$ where $z$ is the midpoint on the geodesic going from $a$ to $b$. The gradient flow of the distance to the diagonal $\Delta_{\mathbb{H}}$ is obtained by ``flowing" simultaneously $a$ and $b$ toward $z$. From this one sees easily that the differential of the function $f$ is nonzero when restricted to each of the two subspaces $E_{y, j}$ ($j=1, 2$).

Let $w\in W\setminus \widehat{Z}$ with ${\rm d}_w\pi \neq 0$. We write $\widehat{q}(w)=(x_{1},x_{2})$. If $d_w\pi$ is a submersion, then $d_w(f\circ \pi)\neq 0$. Thus, we may assume that ${\rm d}_w \pi$ has $\C$-rank $1$. This is the case if exactly one of the $x_i$'s is a critical point of $p_i$, say $x_1$. In that case we have that $\im ({\rm d}_w\pi)=E_{\pi(w),2}$ and therefore ${\rm d}_w(f\circ \pi)\neq 0$ by the previous observation applied to $y=\pi (w)$. This completes the proof of the equivalence of (1), (2) and (3). The condition
$$d_{\pi (w)}f (E_{\pi (w),j})\neq 0\;\;\;\; (j=1, 2)$$ being $C^1$-open, the assertion about perturbations of $f$ follows immediately.
\end{proof}

We now construct the function $h$. Let $${\rm CV}((p_{1},p_{2}))\subset \Sigma \times \Sigma$$ be the product of the set of critical values of $p_1$ with the set of critical values of $p_2$. In other words, ${\rm CV}((p_{1},p_{2}))$ is the image under $(p_{1},p_{2})$ of the set of zeros of the differential of $(p_{1},p_{2})$. This set is finite and does not intersect the diagonal
$$\Delta_{\Sigma}\subset \Sigma \times \Sigma,$$
by assumption. Hence the set ${\rm CV}(\pi)=q^{-1}({\rm CV}((p_{1},p_{2})))$ is discrete and does not intersect $\widehat{\Delta}_{\Sigma}$. By compactness it also avoids a closed neighborhood $U$ of $\widehat{\Delta}_{\Sigma}$. We will perturb $f$ near each point of ${\rm CV}(\pi)$ to obtain a function $f' : Y \to \R_{\ge 0}$ and will later perturb $f'\circ \pi$ near the zero set of the differential of $\pi$ to obtain the function $h$.  

We pick a family of disjoint closed balls $$(B(v))_{v\in {\rm CV}(\pi)},$$ which are also disjoint from $U$. 
We identify $B(v)$ with the closed unit ball $B_{1}$ of $\C^{2}$, via some holomorphic coordinates. We denote by $B_r$ the closed ball of radius $r$ and use the symbol $\vert \vert \cdot \vert \vert$ simultaneously for the Euclidean norm on $\C^{2}$ and for the dual norm on $(\C^{2})^{\ast}$ defined by
$$\vert \vert \varphi \vert \vert = \underset{\vert \vert v \vert \vert \le 1}{{\rm sup}}\vert \varphi (v)\vert,$$
for $\varphi \in (\C^{2})^{\ast}$. We fix once and for all a smooth function
$$\chi : \C^2\to [0,1]$$ such that $\chi$ is equal to $0$ on a neighborhood of the origin, and is equal to $1$ on a neighborhood of the set $\{\vert \vert z \vert \vert \ge 1\}$. We now pick a constant $D\ge 1$ such that the inequalities:
$$\vert \vert d_{z}\chi\vert \vert \le D,$$
$$\vert f(z)- (f(0) + d_{0}f(z))\vert\le D \vert \vert z \vert \vert^{2},$$
$$\vert \vert d_{z}f-d_{0}f\vert \vert \le D \vert \vert z\vert \vert$$ hold for each point $z\in B_1$. For $r\in (0,1)$, we let $$h_{r}(z)=\chi (\frac{z}{r})f(z)+\left(1-\chi (\frac{z}{r})\right)(f(0)+d_{0}f(z)).$$
The function $h_r$ is equal to $f$ near the boundary of $B_r$ and we have on $B_r$:
\begin{equation}
\vert f(z)-h_{r}(z)\vert \le  D r^2
\end{equation}
and
\begin{equation}
d_{z}h_{r}=d_{0}f+E(z,r)
\end{equation}
where the norm of the linear form $E(z,r)$ is bounded by $(D+D^{2})r$. In particular, there exists $r_0>0$ such that for $r\le r_{0}$ the differential $d_{z}h_{r}$ cannot vanish for $z\in B_r$. We also assume that $r_0$ is chosen small enough so that $h_{r}>0$ and so that the critical points of $h_{r}\circ \pi$ on $\pi^{-1}(B_r)$ are the zeros of the differential of $\pi$ (for $r\le r_{0}$); this is possible thanks to Lemma~\ref{lem:crdi}. Since $h_{r_{0}}$ coincides with $f$ near the boundary of $B_{r_{0}}$ we can modify $f$ by replacing it by $h_{r_{0}}$ in the ball $B_{r_{0}}$. We perform this modification in all the balls $(B(v))_{v\in {\rm CV}(\pi)}$ and obtain a new function $f' : Y \to \R_{\ge 0}$. Hence the critical points of the function
$$f'\circ \pi : W \to \R_{\ge 0}$$ are the zeros of $d\pi$, away from $\widehat{Z}$. Let $w\in W \setminus \widehat{Z}$ be such a critical point. We pick coordinates as before on the ball $B(\pi (w))$. Up to a constant, the map $f'$ is equal to a linear form $\ell$ near $\pi (w)$. We write $\ell ={\rm Re}(A)$ where $A : \C^2 \to \C$ is a complex linear form. The holomorphic function $A\circ \pi$ has an isolated critical point at $w$. Let $\mu$ be its Milnor number (see Section 7 and Appendix B in~\cite{Mil-68}).

We now choose coordinates on a small closed ball $B$ around $w$. Let $B' \Subset B$ be a smaller closed ball. We assume that $A\circ \pi$ has only one critical point in $B$ (i.e. $w$, identified with the origin). There exist complex linear forms $u : \C^{n_{1}+n_{2}}\to \C$ arbitrarily close to $0$ so that $A\circ \pi+u$ has no critical point near the boundary of $B$ and $\mu$ nondegenerate critical points in the interior of $B$, all contained in $B'$. We modify $f'\circ \pi$ inside $B$ so that it equals ${\rm Re}(A\circ \pi+u)$ inside $B'$, $f'\circ \pi$ near the boundary of $B$, and so that its only critical point are inside $B'$. We can assume that the perturbation still takes positive values and is at distance at most $1$ from $f'\circ \pi$.   

 Finally, we denote by $h$ the function obtained by modifying $f'\circ \pi$ as above in a neighbourhood of each zero of $d\pi$. Since $h$ is the real part of a holomorphic Morse function in a neighbourhood of each of its critical points, they are all non-degenerate of index $n_1 + n_2$. To summarize, we have proved:

\begin{proposition}
There exists a smooth function 
$$h : W \to \R_{\ge 0}$$ such that:
\begin{enumerate}
\item $h$ coincides with $f\circ \pi$ on a neighborhood of $\widehat{Z}$,
\item $h$ is proper,
\item the set of critical points of $h$ in $W\setminus \widehat{Z}$ is discrete and each critical point is nondegenerate of index $n_{1}+n_{2}$.
\end{enumerate}
\end{proposition}

\subsection{Conclusion of the proof}
\label{sec:ComplexMorse}

We start this section by two propositions which will easily imply Theorem~\ref{thm:FibreProduct}. 

\begin{proposition}\label{prop:attaching-n-cells}
     The inclusion $\widehat{Z} \hookrightarrow W$ induces an isomorphism on homotopy groups up to degree $n_{1}+n_{2}-2$.  
\end{proposition}

\begin{proof} Since $h=f\circ \pi$ close enough to $\widehat{Z}$ and since $f\circ \pi$ has no critical points near $\widehat{Z}$, besides the points of $\widehat{Z}$ themselves, we can pick $\varepsilon >0$ such that the inclusion $\widehat{Z} \hookrightarrow h^{-1}([0,\varepsilon])$ is a homotopy equivalence. This follows from the fact that any tubular neighbourhood of $\widehat{Z}$ contains $h^{-1}([0,\delta])$ for some $\delta >0$ and conversely that for any $\delta >0$, $h^{-1}([0,\delta])$ contains a tubular neighbourhood of $\widehat{Z}$.

Since the set of critical points of $h$ is discrete away from $\widehat{Z}$, and since all these critical points have index $n_{1}+n_{2}$, for every regular values $b\ge a$, the space $h^{-1}([0,b])$ is obtained from $h^{-1}([0,a])$, up to homotopy, by attaching finitely many cells of dimension $n_{1}+n_{2}$. In particular the inclusion $h^{-1}([0,\varepsilon])\hookrightarrow W$ induces an isomorphism on homotopy groups up to degree $n_{1}+n_{2}-2$. Combined with the observation from the previous paragraph, this fact implies the conclusion of the proposition.
\end{proof}

\begin{proposition}\label{prop:infinite-homology}
    The $(n_1+n_2)$-th homology group $H_{n_1+n_2}(W,\mathbb{Q})$ is infinite dimensional.
\end{proposition}

\begin{proof} This is identical to an argument from~\cite{NicPy-21}. We choose an increasing sequence of regular values $(a_{k})_{k\ge 0}$ of $h$ converging to infinity in such a way that $h^{-1}((a_{k},a_{k+1}))$ always contains at least one critical point of $h$. The group $$H_{n_{1}+n_{2}}(W,\Q)$$ is the direct limit of the groups $H_{n_{1}+n_{2}}(h^{-1}([0,a_{k}]),\Q)$. We write $W_{k}:=h^{-1}([0,a_{k}])$. Since $W_{k+1}$ is obtained from $W_{k}$ by gluing $(n_{1}+n_{2})$-dimensional cells, up to homotopy, the Betti numbers
$$(b_{n_{1}+n_{2}-1}(W_{k}))_{k\ge 0}$$ form a non-increasing sequence. We pick an integer $k_{0}$ such that this sequence is constant for $k\ge k_{0}$. An application of the Mayer-Vietoris sequence then proves that the maps
$$H_{n_{1}+n_{2}}(W_{k},\Q)\to H_{n_{1}+n_{2}}(W_{k+1},\Q)\;\;\;\; (k\ge k_{0})$$ 
are injective and that the sequence $(b_{n_{1}+n_{2}}(W_{k}))_{k\ge k_{0}}$ is strictly increasing. This implies that $H_{n_{1}+n_{2}}(W,\Q)$ is infinite dimensional and concludes the proof. We refer the reader to~\cite[p. 61--62]{NicPy-21} for more details. 
\end{proof}

We can now conclude the proof of Theorem \ref{thm:FibreProduct}. 

\begin{proof}[{Proof of Theorem \ref{thm:FibreProduct}}] The first item of the theorem follows from Propositions~\ref{prop:attaching-n-cells} and from the following remark: since $n_{1}+n_{2}\ge 3$, the inclusion $\widehat{Z} \hookrightarrow W$ induces an isomorphism on fundamental group and since $Z$ and $\widehat{Z}$ are diffeomorphic, $\pi_{1}(Z)$ and $\pi_{1}(W)$ indeed have isomorphic fundamental groups. The second item of the theorem follows from Proposition~\ref{prop:infinite-homology}. To prove the third item, we observe that the fundamental group of $W$ is by construction equal to the fiber product:
\begin{equation}\label{eq:fproducts6}
\{(a,b)\in \pi_{1}(X_{1})\times \pi_{1}(X_{2}), p_{1,\ast}(a)=p_{2,\ast}(b)\}.
\end{equation}
 Since the projection from $\widehat{Z}$ to $Z$ is an isomorphism, we obtain that the morphism induced by the inclusion of $Z$ in $X_{1}\times X_{2}$ is injective at the level of fundamental groups, with image given by the subgroup~\eqref{eq:fproducts6}.

Finally, if $X_1$ and $X_2$ are aspherical, so is $W$. This means that $W$ is a $K(\pi_1(W),1)$. Proposition~\ref{prop:infinite-homology} then implies that 
$$\pi_{1}(Z)\cong \pi_{1}(W)$$
is not of type ${\rm FP}_{n_{1}+n_{2}}(\Q)$ while Proposition~\ref{prop:attaching-n-cells} implies that it is of type $\mathscr{F}_{n_{1}+n_{2}-1}$. This completes the proof of Theorem \ref{thm:FibreProduct}. 
\end{proof}


\section{Subgroups of $n$-iterated Kodaira fibrations}
\label{sec:Koc-Vid}
The goal of this section is to use methods from complex geometry to provide a new proof of a result of Kochloukova and Vidussi on finiteness properties of subgroups of fundamental groups of $n$-iterated Kodaira fibrations. Along the way, our proof shows that certain iterated Kodaira fibrations admit finite Albanese maps. This provides more examples of closed aspherical K\"ahler manifolds to which the methods of the present article and of~\cite{LloPy-22} can be applied. 

\subsection{Constructing $n$-iterated Kodaira fibrations}
We start by recalling the inductive definition of $n$-iterated Kodaira fibrations.
\begin{definition}
 Let $X$ be a compact complex manifold of dimension $n\geq 1$. If $n=2$, then we call $X$ a \emph{$2$-iterated Kodaira fibration} (or simply a \emph{Kodaira fibration}) if there exists a holomorphic submersion $\pi: X\to Y$ with connected fibres onto a closed hyperbolic Riemann surface, which is not isotrivial. If $n>2$, then we call $X$ an \emph{$n$-iterated Kodaira fibration} if there is a holomorphic submersion $\pi:X\to Y$ with connected fibres onto an $(n-1)$-iterated Kodaira fibration $Y$, which is not isotrivial.
\end{definition}

We call a group $G$ a \emph{polysurface group} of length $n$ if there is a filtration
\[
    1=G_0\unlhd G_1\unlhd G_2\unlhd \dots \unlhd G_n=G
\]
such that $G_i/G_{i-1}$ is the fundamental group of an orientable closed hyperbolic surface for $1\leq i \leq n$. By construction fundamental groups of $n$-iterated Kodaira fibrations are polysurface groups. We call a group $G$ \emph{irreducible} if no finite index subgroup of $G$ decomposes as a direct product of two non-trivial groups.

In \cite{LloPy-21}, we inductively constructed $n$-iterated Kodaira fibrations with injective monodromy, starting from the Kodaira--Atiyah examples of ($2$-iterated) Kodaira fibrations, which are known to have injective monodromy. The induction step is given by the following theorem, summarising results from \cite[Section 5]{LloPy-21}.
\begin{theorem}\label{thm:ItKodFib}
    Let $X$ be an $n$-iterated Kodaira fibration with injective monodromy. Then there are finite coverings $X',~ X''\to X$ such that $X'$ is the base space of an $(n+1)$-iterated Kodaira fibration $Z\to X'$ with injective monodromy and there is a finite map $Z\to X'\times X''$ which defines a ramified covering of its image.
\end{theorem}
The proof of Theorem \ref{thm:ItKodFib} consists in performing a classical construction due to Kodaira and Atiyah ``in family". It also relies on the fact that a finite cover $X'\to X$ of an $n$-iterated Kodaira fibration $X$ with injective mondromy is again an $n$-iterated Kodaira fibration with injective monodromy, with base a finite cover of the base of $X$ (see the proof of \cite[Proposition 39]{LloPy-21}). In particular, if one starts from the Kodaira--Atiyah fibration, the following is an implicit consequence of the construction in \cite[Section 5]{LloPy-21}.

\begin{theorem}\label{thm:ItKodFib2}
Fix an integer $n\geq 2$. There exists a sequence of $i$-iterated Kodaira fibrations $X^i$, $2\leq i \leq n$, and a closed hyperbolic Riemann surface $X^1$, together with holomorphic submersions with connected fibres $\pi_{i,i-1}: X^i\to X^{i-1}$, which are not isotrivial and have injective monodromy, with the following properties: 
\begin{enumerate}
\item there is a finite map $f: X^n\to S_1\times \dots \times S_{2^{n-1}}$ to a direct product of $2^{n-1}$ closed hyperbolic Riemann surfaces,
\item the group $\pi_1(X^n)$ is irreducible.
\end{enumerate}
\end{theorem}
\begin{proof}
    By construction \cite{Ati-69,Kod-67}, the Kodaira-Atiyah fibration $X^2$ is a ramified cover of a direct product $R\times T$ of two closed hyperbolic Riemann surfaces and the projection to either of the factors is a holomorphic submersion with connected fibres and injective monodromy. In particular, the map $X^2\to R\times T$ is finite and we can choose $X^1=R$. Inductively applying Theorem \ref{thm:ItKodFib} and passing to finite covers of the $X^i$ if necessary, we can now construct a sequence of $i$-iterated Kodaira fibrations with the desired properties. The only thing that is not an immediate consequence is the irreducibility statement. For this we will prove by induction on $i$ the following stronger statement: if $G_{1}$ and $G_{2}$ are two commuting subgroups of $\pi_{1}(X^{i})$ such that $G_{1}\cdot G_{2}$ has finite index in $\pi_{1}(X^{i})$, then either $G_1$ or $G_2$ is trivial. When $i=2$ the argument is essentially contained in~\cite{Joh-1994}; we give a complete proof however. We first observe that the previous statement is true in a surface group, implying the case $i=1$, as the reader can check readily. Let now $i\in \{2, \ldots , n\}$ and let $G_1$ and $G_2$ be commuting subgroups of $\pi_{1}(X^{i})$, generating a subgroup of finite index. We consider the fibration
    $$\pi_{i,i-1} : X^{i}\to X^{i-1}.$$
    By induction hypothesis one of the two groups $$\pi_{i,i-1}(G_{1}), \;\; \pi_{i,i-1}(G_{2})$$ is trivial\footnote{If $i=2$, this also follows because $\pi_1(X^{1})$ is a surface group.}. We assume that $\pi_{i, i-1}(G_{1})=1$. The group $H:={\rm Ker}(\pi_{i,i-1 \ast})\cap \left(G_{1}\cdot G_{2}\right)$ has finite index in ${\rm Ker}(\pi_{i,i-1 \ast})$, hence is a surface group. The group $G_1$ is normal in $H$ and $G_2$ normalizes both $H$ and $G_{1}$. Since $G_2$ centralizes $G_1$, Lemma 35 from~\cite{LloPy-21} implies that $G_2$ actually centralizes $H$. This in turn implies that $G_{2}$ centralizes ${\rm Ker}(\pi_{i,i-1 \ast})$. Hence $\pi_{i,i-1 \ast}|_{G_2}$ is injective and its image is contained in the kernel of the monodromy representation. This contradicts the injectivity of this representation and concludes the proof of our statement.\end{proof}

\subsection{A complex geometry proof of a Theorem of Kochloukova--Vidussi}

We now fix an integer $n\ge 2$ and a sequence $X^i$ as in Theorem \ref{thm:ItKodFib2}. For $i\le j$, we denote by $\pi_{j,i} : X^j\to X^i$ the surjective holomorphic maps with connected fibres obtained by composing the various maps $X^{s}\to X^{s-1}$. For $2\leq i \leq n$, let $F_i$ be a fibre of $\pi_{i,i-1}$. The inclusion $F_i\hookrightarrow X^i$ is $\pi_1$-injective. We deduce that the restrictions $\pi_{j,j-1}: \pi_{j,i}^{-1}(F_i)\to \pi_{j-1,i}^{-1}(F_i)$ define $(j-i+1)$-iterated Kodaira fibrations with injective monodromy for $n\geq j\geq i+1$. We then define $$Y^{i}:=\pi_{n,n-i+1}^{-1}(F_{n-i+1})\subseteq X^{n},$$ and observe that the fundamental group of $Y^i$ injects into that of $X^n$. From this, we deduce the following result.

\begin{proposition}\label{prop:ItKodaira}
    There exists an $n$-iterated Kodaira fibration $X^n$ with the following properties:
    \begin{enumerate}
        \item $X^n$ has injective monodromy and for every integer $i$ with $n\geq i \geq 2$, $X^n$ contains a $i$-iterated Kodaira fibration $Y^i\subseteq X^n$, which is $\pi_{1}$-injected.
        \item The group $\pi_1(X^n)$ is irreducible.
        \item There is a finite holomorphic map $X^n\to A$ to a complex torus.
    \end{enumerate}
\end{proposition}
\begin{proof}
    We choose $X^n$ as in Theorem \ref{thm:ItKodFib2}. The first and second part then follow from the above discussion and Theorem \ref{thm:ItKodFib2}. Moreover, there is a finite holomorphic map $X^n\to S_1\times \dots \times S_{2^{n-1}}$ to a direct product of $2^{n-1}$ closed hyperbolic Riemann surfaces. Since the Albanese map of a closed hyperbolic Riemann surface is an embedding, choosing $A$ to be the Albanese torus of $S_1\times \dots\times S_{2^{n-1}}$ completes the proof.
\end{proof}

 As a consequence of Proposition \ref{prop:ItKodaira}, we obtain a new proof of a result of Kochloukova and Vidussi \cite[Corollary 1.11]{KocVid-22}.
\begin{proof}[{Proof of Theorem \ref{thm:KV}}]
    Let $X^n$ be an $n$-iterated Kodaira fibration satisfying the conditions in Proposition \ref{prop:ItKodaira}. The smooth projective variety $X^n$ is obtained by iteratively constructing locally trivial surface bundles starting from a closed Riemann surface. In particular, $X^n$ is aspherical and, by multiplicativity of the Euler characteristic, its Euler characteristic is non-trivial. The same applies for the $i$-iterated Kodaira fibrations $Y^i\subseteq X^n$. Since, moreover, the restriction of the finite holomorphic map $X^n \to A$ to any of the $Y^i$ is finite holomorphic, all of the $Y^i$ satisfy the hypotheses of Theorem \ref{thm:LP-BNSR}. Thus, for every integer $i\in \left\{2,\dots,n\right\}$, there is a character $\chi_i:\pi_1(Y_i)\to \mathbb{Z}$ with kernel of type $\mathscr{F}_{i-1}$ and not of type ${\rm FP}_i(\mathbb{Q})$. Finally, for $i=1$ the assertion follows, since the fundamental group of the fibre of $\pi_{n,n-1}$ is a hyperbolic surface group.
\end{proof}

\begin{remark}
If $\Gamma < {\rm PU}(n,1)$ is an arithmetic lattice of the simplest type such that the Albanese map of the manifold $X=\Gamma\backslash \mathbb{B}^{n}_{\C}$ is finite, it is also true (in analogy with Theorem~\ref{thm:KV}) that for each $i\in \{1, \ldots , n\}$, $\Gamma$ contains a subgroup of type $\mathscr{F}_{i-1}$ which is not of type ${\rm FP}_{i}(\Q)$. This follows from applying the result from the present article (or from~\cite{LloPy-22}) to the fundamental group of suitable complex totally geodesic submanifolds, embedded in $X$ or some finite covering space of $X$.     
\end{remark}


\section{Further remarks}\label{morse-theoretic-remarks}

As mentioned in the introduction, we explain in this section how to recover some of our results (and results from~\cite{LloPy-22}) using purely real Morse theoretical arguments instead of the theory of BNSR-invariants. This concerns mainly the study of kernels of homomorphisms to $\Z$. For homomorphisms to $\Z^2$, there are no purely Morse theoretical arguments known to us. 

We recall that a closed $1$-form on a manifold is said to be Morse if it is locally the differential of a Morse function. We start with the following:

\begin{proposition}\label{prop:morsebnsr} Let $X$ be a closed oriented aspherical $2n$-dimensional manifold. Let $O_{n}\subset H^{1}(X,\R)-\{0\}$ be the set of nontrivial classes which can be represented by a Morse $1$-form all of whose zeros have index $n$. Then:
\begin{enumerate}
\item The set $O_n$ is open. 
\item The projection of $O_n$ in the sphere $S(\pi_{1}(X))$ is contained in $\Sigma^{n-1}(\pi_{1}(X))\cap -\Sigma^{n-1}(\pi_{1}(X))$. 
\item If $\xi$ is a rational class in $O_n$, the kernel of $\xi$ is of type $\mathscr{F}_{n-1}$. 
\end{enumerate}
\end{proposition}

\begin{proof} Let $a$ be a class in $O_n$ represented by a Morse $1$-form $\theta$, all of whose zeros having index $n$. Let $\theta_{1}, \ldots , \theta_{N}$ be closed $1$-forms whose cohomology classes form a basis of $H^{1}(X,\R)$. There exists a neighbourhood $U$ of the origin in $\R^N$ such that if $(x_{1}, \ldots , x_{N})\in U$, then the form 
$$\theta +\sum_{j=1}^{N}x_{j}\theta_{j}$$
is Morse and all its zeros have index $n$. Hence the classes 
\begin{equation}\label{eq:morseforms}
\bigg[\theta +\sum_{j=1}^{N}x_{j}\theta_{j}\bigg], \;\;\;\; (x_{1}, \ldots ,x_{N})\in U
\end{equation}
lie in $O_n$. Since the classes in~\eqref{eq:morseforms} form a neighbourhood of $a$, this proves the first claim. 

The second claim can be shown using arguments from Morse theory; we refer to Section 2.3 in~\cite{LloPy-22} for a proof that is based on ideas from the work of Simpson~\cite{Sim-93}. The proof given there is stated for a closed form which is the real part of a holomorphic $1$-form with finitely many zeros. It applies equally well for a closed (real) Morse $1$-form all of whose zeros have index equal to half the real dimension of the manifold. 

We finally prove the third claim, using classical arguments which are similar to the ones used in the proof of Proposition \ref{prop:attaching-n-cells}. We assume that $n\ge 2$, otherwise there is nothing to prove. If $\xi \in O_n$ is rational, there exists a closed $1$-form $\alpha$ which is Morse and whose critical points are all of index $n$, whose cohomology class is proportional to $\xi$, and such that the integration morphism 
\begin{equation}\label{eq:intmo}
I_{\alpha} : \pi_{1}(X)\to \R
\end{equation}
 defined by $\alpha$ has image equal to $\Z$. Let $\widehat{X}\to X$ be the covering space associated to the kernel of $I_{\alpha}$. We fix a primitive $f : \widehat{X}\to \R$ of the lift of $\alpha$ to $\widehat{X}$. The function $f$ is a proper Morse function, all of whose critical points have index $n$. Let $c$ be a regular value of $f$ and let $\widehat{X}_{c}:=f^{-1}(c)$. By real Morse theory, for all positive real numbers $t$ such that $c\pm t$ is a regular value of $f$, the space $f^{-1}([c-t, c+t])$ is obtained from the closed manifold $f^{-1}(c)$ up to homotopy equivalence by attaching finitely many $n$-cells. Since $f^{-1}(c)$ can be equipped with a finite CW-complex structure, this implies that the space $\widehat{X}$ is homotopy equivalent to a CW-complex with finite $(n-1)$-skeleton. Since $\widehat{X}$ is a classifying space for ${\rm ker}(\chi)$, we deduce that this group is of type $\mathscr{F}_{n-1}$. This concludes the proof of Proposition~\ref{prop:morsebnsr}.
 \end{proof}

To continue our discussion, we shall need the following classical result. 

\begin{proposition}\label{prop:realmorsety} Let $X$ be a closed complex manifold of complex dimension $n$ and let $\alpha$ be a closed holomorphic $1$-form with finitely many zeros. Then the cohomology class $[{\rm Re}(\alpha)]$ can be represented by a Morse $1$-form all of whose zeros have index $n$. 
\end{proposition}

The proof of this proposition is the same as the second part of the deformation argument in Section \ref{sec:perturbation}: near each point $p$ which is a zero of $\alpha$, we write $\alpha = dh$ for some holomorphic function $h$ defined on a ball $B_p$ centered at $p$. We can perturb $h$ into a holomorphic function $h' : B_p \to \C$ which has finitely many critical points, all nondegenerate, and which has no critical point near the boundary. We then take one more $C^{\infty}$ perturbation $h_p$ which coincides with $h$ near the boundary of the ball and with $h'$ on a smaller ball, in such a way that $h_p=h'$ near each critical point of $h_p$. We consider the $C^{\infty}$ complex-valued form $\beta$ which equals $dh_p$ in $B_p$ and $\alpha$ outside of the union of the balls $B_p$. The forms $\alpha$ and $\beta$ are cohomologous, hence so are their real parts. The only zeros of the real part of $\beta$ are contained in one of the balls $B_p$ (for $p$ a zero of $\alpha$) and are critical points of $h_p$. Since $h_p$ is a holomorphic Morse function near its critical points, the claim follows, observing that the real part of the function $$(z_{1},\ldots , z_{n})\mapsto z_{1}^{2}+\cdots +z_{n}^{2}$$
is given by $(x_{1}^{2}+\cdots +x_{n}^{2})-(y_{1}^{2}+\cdots +y_{n}^{2})$ if $z_j=x_j +iy_j$ with $x_{j}, y_{j}\in \R$. We refer the reader to~\cite[p. 59-60]{NicPy-21} for more details on the perturbation argument.

As announced above, the following proposition can serve as a substitute for the use of BNSR-invariants, when studying kernels of homomorphism to $\Z$. Instead of working with the set $\Sigma^{n-1} \cap - \Sigma^{n-1}$, we simply consider the smaller set $O_n$, which is also open, and which is dense under suitable hypotheses.   

\begin{proposition}\label{prop:subsproof} Let $X$ be a closed K\"ahler manifold of complex dimension $n$, and let $O_n \subset H^{1}(X,\R)-\{0\}$ be defined as in Proposition~\ref{prop:morsebnsr}. Assume that $X$ has finite Albanese map. Then: 
\begin{enumerate}
\item The set $O_n$ is dense in $H^{1}(X,\R)-\{0\}$,
\item if moreover $X$ is aspherical, any rational class in $O_n$ has kernel of type $\mathscr{F}_{n-1}$ (and not of type ${\rm FP}_{n}(\mathbb{Q})$ if the Euler characteristic of $X$ is nonzero).
\end{enumerate} 
\end{proposition}

\begin{proof} Let $a_X : X \to A(X)$ be the Albanese map of $X$. If $\beta$ is a holomorphic $1$-form on $A(X)$ whose restriction to any positive dimensional subtorus does not vanish, then $\alpha =a_X^{\ast}\beta$ has finitely many zeros; moreover the set of forms $\alpha$ obtained in this way is dense in $H^{1,0}(X)$ (see Propositions 14 and 18 in~\cite{LloPy-22}). According to Proposition~\ref{prop:realmorsety}, we have $[{\rm Re} (\alpha)]\in O_n$ for such an $\alpha$. This proves that $O_n$ is dense. The second point follows from Proposition~\ref{prop:morsebnsr} and the last claim in Addendum~\ref{add-to-thm:LP-BNSR}. 
\end{proof}

We now turn to the existence of {\it perfect circle-valued Morse functions} on K\"ahler manifolds. Let us recall the terminology first. Let $M$ be a smooth closed manifold. We call a map $M \to S^1$ a circle-valued function and say that it is Morse if it coincides locally with a Morse function. A circle-valued Morse function $f : M \to S^1$ allows to study the topology of $M$ by starting with a regular fibre of $f$, thickening it, and attaching handles when passing a critical level set. In particular for such an $f$ we have the formula:
$$\chi (M)=\sum_{x} (-1)^{{\rm ind}(x)},$$
where the sum runs over the finitely many critical points of $f$, and ${\rm ind}(x)$ is the index of a critical point $x$. We say that $f$ is perfect if it has $\vert \chi (M)\vert$ critical points. When $M$ is odd-dimensional, or simply of Euler characteristic equal to $0$, this means that $f$ is a fibration. When $M$ is even-dimensional of nonzero Euler characteristic, this happens if and only if the indices of the critical points all have the same parity. Some existence results for perfect circle-valued Morse functions on negatively curved manifolds are available: every closed hyperbolic $3$-manifold has a finite cover fibring over the circle, as follows from the work of Agol~\cite{Ago-13}, see also~\cite{Ber-15}; in dimension $5$, Italiano--Martelli--Migliorini built examples of cusped real hyperbolic manifolds fibring over the circle~\cite{IMM-22}, and in dimension $4$, Battista and Martelli found finitely many examples (both closed and cusped) of real hyperbolic $4$-manifolds admitting perfect circle-valued Morse functions~\cite{BaMa-22}. We observe here that Propositions~\ref{prop:realmorsety} and~\ref{prop:subsproof} have the following immediate consequence. 

\begin{theorem} Let $X$ be a closed K\"ahler manifold of complex dimension $n$, with finite Albanese map. Then any rational ray contained in the dense open set $O_n\subset H^{1}(X,\R)-\{0\}$ can be represented by a circle-valued Morse function all of whose critical points have index $n$ (in particular such a Morse function is perfect). 
\end{theorem}

Applying the result of Eyssidieux~\cite{Eys-18} already alluded to before (see also~\cite{LloPy-22}), we obtain:

\begin{corollary} Let $\Gamma < {\rm PU}(n,1)$ be a cocompact torsion-free arithmetic lattice. Assume that $b_{1}(\Gamma)>0$. Then there is a finite index subgroup $\Gamma_0 < \Gamma$ such that for any subgroup of finite index $\Gamma_1 < \Gamma_0$, there is a dense open set $O\subset H^{1}(\Gamma_1,\R)$ such that any rational class in $O$ can be represented (up to scalar) by a circle-valued Morse function on $\Gamma \backslash \mathbb{B}_{\mathbb{C}}^n$ all of whose critical points have index $n$. 
\end{corollary}

\bibliography{References}
\bibliographystyle{amsplain}

\end{document}